\documentclass[11pt]{amsart}

\usepackage[T1]{fontenc}
\usepackage{lmodern}
\usepackage{microtype}
\usepackage{amsmath,amssymb,amsthm,mathtools}
\usepackage[margin=1in]{geometry}
\usepackage{enumitem}
\usepackage{booktabs}
\usepackage{aliascnt}
\usepackage[colorlinks=true,linkcolor=blue,citecolor=blue,urlcolor=blue]{hyperref}
\usepackage[nameinlink,capitalise]{cleveref}

\newtheorem{theorem}{Theorem}[section]

\newaliascnt{lemma}{theorem}
\newtheorem{lemma}[lemma]{Lemma}
\aliascntresetthe{lemma}
\crefname{lemma}{lemma}{lemmas}
\Crefname{lemma}{Lemma}{Lemmas}

\newaliascnt{proposition}{theorem}
\newtheorem{proposition}[proposition]{Proposition}
\aliascntresetthe{proposition}
\crefname{proposition}{proposition}{propositions}
\Crefname{proposition}{Proposition}{Propositions}

\newaliascnt{corollary}{theorem}
\newtheorem{corollary}[corollary]{Corollary}
\aliascntresetthe{corollary}
\crefname{corollary}{corollary}{corollaries}
\Crefname{corollary}{Corollary}{Corollaries}

\theoremstyle{definition}
\newaliascnt{definition}{theorem}
\newtheorem{definition}[definition]{Definition}
\aliascntresetthe{definition}
\crefname{definition}{definition}{definitions}
\Crefname{definition}{Definition}{Definitions}

\theoremstyle{remark}
\newaliascnt{remark}{theorem}
\newtheorem{remark}[remark]{Remark}
\aliascntresetthe{remark}
\crefname{remark}{remark}{remarks}
\Crefname{remark}{Remark}{Remarks}

\numberwithin{equation}{section}

\newcommand{\N}{\mathbb N}
\newcommand{\rad}{\operatorname{rad}}
\newcommand{\e}{\mathrm e}
\newcommand{\Acal}{\mathcal A}
\newcommand{\Pcal}{\mathcal P}
\newcommand{\Bcal}{\mathcal B}
\newcommand{\mathbfone}{\mathbf 1}
\newcommand{\eps}{\varepsilon}
\newcommand{\AIDisclosure}{%
\medskip
\noindent\textbf{Use of artificial intelligence tools:}
Large language models, primarily OpenAI's ChatGPT, were used extensively
throughout the research. The author originated the ideas and research
directions, while the models contributed substantially to the technical
development of the work: they were used to explore and further develop the
ideas, produce technical lemmas and calculations, generate and debug code,
and assist in auditing the proof. The author critically evaluated outputs
from different model instances, selected and synthesised promising elements,
and discarded or corrected unsuccessful or flawed suggestions. The Lean
formalisation~\cite{Li2026lean} was produced with Harmonic's Aristotle,
directed and audited by the author. All strategic decisions were made by the
author, who takes full responsibility for the mathematical correctness of
the final results.\par
\medskip}

\title[Erd\H{o}s Problem 768]{A Resolution of Erd\H{o}s Problem 768: the Sylow Divisor Condition}

\author[E. Li]{Eric Li}
\dedicatory{\normalfont\normalsize Trinity College, University of Cambridge}
\date{13 July 2026}
\thanks{Email addresses: \href{mailto:el593@cam.ac.uk}{el593@cam.ac.uk}, \href{mailto:contact@ericli.com}{contact@ericli.com}.}


\hypersetup{
  pdftitle={A Resolution of Erd\H{o}s Problem 768: the Sylow Divisor Condition},  pdfauthor={Eric Li},
  pdfsubject={Resolution of Erd\H{o}s Problem 768 and sharp asymptotic for the Sylow divisor condition},
  pdfkeywords={Sylow divisor condition, Erd\H{o}s Problem 768, multiplicative large sieve, subset products, divisor moments}
}

\begin{document}

\begin{abstract}
We resolve Erd\H{o}s Problem~768.  Let $A(x)$ count the positive integers
$n\le x$ such that, for every prime $p\mid n$, there is a divisor $d>1$ of
$n$ with $d\equiv1\pmod p$.  Erd\H{o}s asked whether
\[
 \frac{A(x)}x
 =\exp\!\bigl(-(c+o(1))\sqrt{\log x}\log\log x\bigr)
\]
for some constant $c>0$.  We prove that the limit exists and that
\[
 \lim_{x\to\infty}
 \frac{\log(x/A(x))}{\sqrt{\log x}\log\log x}
 =\frac1{2\sqrt{\log 2}}.
\]
Equivalently, the conjectural asymptotic holds with
$c=1/(2\sqrt{\log 2})$.  The lower bound is obtained from primes in disjoint
logarithmic intervals using a fourth-moment argument based on the
multiplicative large sieve and a subset-product second moment.  The upper
bound uses canonical witness divisors, a deterministic compression map, an
injective reconstruction theorem for its fibers, and growing divisor moments.
Thus the paper determines the exact leading constant in Erd\H{o}s Problem~768. 
The main theorem and its complete proof have been formally verified in the
Lean~4 proof assistant.
\end{abstract}

\maketitle

\section{Introduction}

Erd\H{o}s Problem~768 asks for the density decay of integers satisfying a
simple divisor-witness condition motivated by Sylow's theorems.  Let
$\N=\{1,2,3,\ldots\}$.  An integer $n\in\N$ is said to satisfy the
\emph{Sylow divisor condition} if, for every prime $p\mid n$, there is a
divisor $d\mid n$ such that
\begin{equation}\label{eq:def-condition}
 d>1\qquad\text{and}\qquad d\equiv1\pmod p.
\end{equation}
Let $\Acal$ denote the set of such integers, with $1\in\Acal$ by vacuous
truth, and write
\[
 A(x):=\#\{n\le x:n\in\Acal\}.
\]
Erd\H{o}s asked whether the proportion $A(x)/x$ has the form
\begin{equation}\label{eq:erdos-conjecture-intro}
 \frac{A(x)}x
 =\exp\!\bigl(-(c+o(1))\sqrt{\log x}\log\log x\bigr)
\end{equation}
The main result of this paper resolves
Erd\H{o}s Problem~768 by proving that the limit exists and by identifying
this constant exactly.\footnote{The proof of Theorem~\ref{thm:main} has
been formally verified in the Lean~4 proof assistant; see the end of this
section and \cite{Li2026lean}.}

\begin{theorem}\label{thm:main}
One has
\[
 \lim_{x\to\infty}
 \frac{\log(x/A(x))}{\sqrt{\log x}\log\log x}
 =\frac1{2\sqrt{\log 2}}.
\]
Equivalently,
\[
 A(x)=x\exp\!\left(-\left(\frac1{2\sqrt{\log 2}}+o(1)\right)
 \sqrt{\log x}\log\log x\right)
 \qquad (x\to\infty).
\]
\end{theorem}

Thus \cref{thm:main} gives a complete answer to Erd\H{o}s Problem~768: the conjectural form \eqref{eq:erdos-conjecture-intro} holds, and the constant is $c=1/(2\sqrt{\log 2})$.

The terminology comes from Sylow's theorems.  If $G$ is a nonabelian finite
simple group and $p^a\Vert |G|$, then the number of Sylow $p$-subgroups
divides $|G|/p^a$, hence also $|G|$; it is congruent to $1$ modulo $p$ and
is not equal to one.  Therefore $|G|\in\Acal$, so $A(x)$ is an elementary
upper bound for the number of possible orders at most $x$ of nonabelian
finite simple groups.

Dornhoff proved that the set of finite simple-group orders has density zero,
and Dornhoff and Spitznagel obtained a quantitative refinement
\cite{Dornhoff1968,DornhoffSpitznagel1968}; see also
\cite{SpitznagelSzygenda1968,HurleyRudvalis1977}.  Erd\H{o}s later introduced
the present set under the name $V$, together with the larger set $U$ in
which the witness condition is imposed only on the largest prime factor
\cite[p.~198]{Erdos1974}.  In that article, equation~(2) is proved, while
the sharper formula~(6) is announced with details suppressed.  The printed
discussion of~(7) is inconsistent: the display is introduced as something
that can be shown, whereas the following sentence says that it had not been
proved.  We do not use~(7).

A short argument of Sawin shows that $\Acal$ has natural density zero by
considering only the largest prime factor \cite{Sawin2021}; the same set is
recorded as OEIS~A352287 \cite{OEISA352287}.  The current Erd\H{o}s Problems
page lists Problem~768 as open while warning that its bibliography may be
incomplete \cite{ErdosProblems768}.  Related random subset-sum questions in
a different full-coverage model have recently been studied by Ma and Tang
\cite{MaTang2026}.

\AIDisclosure

\subsection*{Outline of the proof and the constant}
The proof of Erd\H{o}s Problem~768 consists of independent lower- and upper-bound arguments that meet at the same constant.

\paragraph{The lower bound.}
We select one prime from each of $r$ disjoint logarithmic intervals.  A
fourth-moment argument based on the multiplicative large sieve removes a
negligible collection of exceptional primes and makes every remaining
source interval Fourier-uniform modulo every remaining target prime.  A
subset-product second moment then shows that almost every selected tuple
supplies, for each of its prime factors, a divisor satisfying
\eqref{eq:def-condition}.  The selected integer has
\[
 \log n=(\log 2+o(1))r^2,
\]
while the logarithm of the counting loss is
\[
 r\log r+r\log\log r+O(r).
\]
Consequently,
\[
 r\log r
 =\left(\frac1{2\sqrt{\log 2}}+o(1)\right)
 \sqrt{\log n}\log\log n,
\]
which yields the required constructive lower bound.

\paragraph{The upper bound.}
For each prime factor of $n$, choose a canonical witness divisor.  A repeated
majority-halving procedure extracts a binary homogeneous subsequence of the
prime factors.  From half of that subsequence we construct a squarefree
factor $Q(n)$ and a deterministic compression map $n\mapsto n/Q(n)$.  The
central reconstruction theorem proves that the fibers of this map are
bounded by growing divisor moments.  If $t=\omega(n)$ and
$\lambda=t/\sqrt{\log x}$, two exponential rates result:
\[
 \frac{\lambda}{2}
 \qquad\text{and}\qquad
 \frac{\lambda}{4}+\frac1{4\lambda\log 2}.
\]
Taking the better estimate and optimizing gives
\begin{equation}\label{eq:intro-optimization}
 \inf_{\lambda>0}\max\left\{
 \frac{\lambda}{2},
 \frac{\lambda}{4}+\frac1{4\lambda\log 2}
 \right\}
 =\frac1{2\sqrt{\log 2}},
\end{equation}
with equality at $\lambda=1/\sqrt{\log 2}$.  This matches the lower-bound
constant.

The logical structure may be summarized as
\[
\begin{gathered}
 \text{prime intervals and the large sieve}
 \longrightarrow \text{Fourier-clean prime layers}\\
 \longrightarrow \text{subset products}
 \longrightarrow \text{lower bound},\\[1mm]
 \text{canonical witnesses}
 \longrightarrow \text{compression and reconstruction}\\
 \longrightarrow \text{growing divisor moments}
 \longrightarrow \text{upper bound}.
\end{gathered}
\]
\subsection*{Formal verification}
Theorem~\ref{thm:main} has been formalised and machine-verified in the
Lean~4 proof assistant (toolchain \texttt{v4.28.0}, Mathlib
\texttt{v4.28.0})~\cite{Li2026lean}.  The formal statement,
\texttt{Erdos768.erdos\_768}, matches Theorem~\ref{thm:main} verbatim,
carries no hypotheses, and depends only on Lean's standard axioms
(\texttt{propext}, \texttt{Classical.choice}, \texttt{Quot.sound}), with no
\texttt{sorry} in its dependency graph; the axiom certificate is enforced at
build time and re-verified by continuous integration.  The multiplicative
large sieve (Theorem~\ref{thm:large-sieve}), the subset-product second
moment (Lemma~\ref{lem:subset-product}), and the
compression--reconstruction package (Propositions~\ref{prop:canonical-fiber}
and~\ref{prop:weighted-prefix}) are proved from first principles over
Mathlib; the analytic input (Lemma~\ref{lem:pnt-layer}) is derived from the
\texttt{MediumPNT} theorem of the PrimeNumberTheoremAnd project, with error
exponent $(\log y)^{1/10}$ in place of $\sqrt{\log y}$, which suffices for
the argument.  The formalisation was produced with Harmonic's Aristotle,
directed and audited by the author, and is archived with process metadata
and complete run histories~\cite{Li2026lean}.

\section{Notation and elementary estimates}

All logarithms are natural unless a base is displayed explicitly; in
particular, $\log_2$ denotes the base-two logarithm.  All implied constants
are absolute unless a dependence is indicated.  For $n>1$, let $P^+(n)$
denote the largest prime factor of $n$.  We write $\chi_0$ for the principal
Dirichlet character, $\omega(n)$ for the number of distinct prime factors
of $n$, and $\tau(n)$ for the number of divisors of $n$.  We put
\[
 \rad(n):=\prod_{p\mid n}p,
\]
and write $\mu$ for the M\"obius function and $\mathbfone(n)=1$ for the
constant arithmetic function.  For a positive integer $k$, $d_k(n)$
denotes the ordered $k$-fold divisor function,
\[
 d_k(n)=\#\{(n_1,\ldots,n_k)\in\N^k:n_1\cdots n_k=n\}.
\]
Thus $d_k(p^\nu)=\binom{\nu+k-1}{\nu}$.

We begin with estimates used in the upper bound.
\begin{lemma}[A uniform divisor-sum bound]\label{lem:dk-sum}
For all real $X\ge1$ and integers $k\ge1$,
\[
 \sum_{n\le X}d_k(n)\le X(1+\log X)^{k-1}.
\]
\end{lemma}

\begin{proof}
The left-hand side counts ordered $k$-tuples of positive integers with product at most $X$.  Summing the last coordinate trivially gives
\[
 \sum_{n\le X}d_k(n)
 \le X\sum_{n_1,\ldots,n_{k-1}\le X}\frac1{n_1\cdots n_{k-1}}
 \le X\left(\sum_{n\le X}\frac1n\right)^{k-1},
\]
and $\sum_{n\le X}1/n\le1+\log X$.
\end{proof}

\begin{lemma}[Growing divisor moments]\label{lem:local-moment}
Let $H\ge0$ be an integer, put $J=2^H$, and let $z\ge1$.  If
\[
 K=\lceil zJ\rceil,
\]
then, for every positive integer $n$,
\[
 z^{\omega(n)}\tau(n)^H\le d_K(n).
\]
In particular, $\tau(n)^H\le d_{2^H}(n)$.
\end{lemma}

\begin{proof}
It is enough to prove the assertion at a prime power $p^\nu$.  The map
\[
 (e_1,\ldots,e_H)\in\{0,1,\ldots,\nu\}^H
\]
can be encoded by the multiplicities of the $2^H$ subsets
\[
 I_j:=\{i:e_i\ge j\},\qquad 1\le j\le\nu.
\]
If $m_I$ is the multiplicity with which the subset $I\subseteq\{1,\ldots,H\}$ occurs among the $I_j$, then the original vector is recovered from
\[
 e_i=\sum_{I\ni i}m_I.
\]
Thus this is an injection into weak compositions of $\nu$ into $J=2^H$ parts, and therefore
\[
 (\nu+1)^H\le\binom{\nu+J-1}{\nu}.
\]
For $\nu\ge1$,
\[
 \frac{\binom{\nu+K-1}{\nu}}{\binom{\nu+J-1}{\nu}}
 =\prod_{j=0}^{\nu-1}\frac{K+j}{J+j}
 \ge\frac KJ\ge z,
\]
because the first factor is $K/J$ and all remaining factors are at least one.  Hence
\[
 z(\nu+1)^H\le\binom{\nu+K-1}{\nu}.
\]
Multiplication over the prime powers dividing $n$ proves the result.
\end{proof}

\begin{lemma}[A restricted growing-moment estimate]\label{lem:restricted-moment}
Let $X\ge3$, let $H\ge0$ be an integer, put $J=2^H$, and let $u>0$ be
real.  Write
\[
 \ell_X:=\log(1+\log X),\qquad B:=J\ell_X.
\]
If $u>B$, then
\[
 \sum_{\substack{m\le X\\ \omega(m)\ge u}}\tau(m)^H
 \le X\exp\!\left(-u\log\frac uB+u\right).
\]
\end{lemma}

\begin{proof}
For $z\ge1$, \cref{lem:local-moment} and \cref{lem:dk-sum} give
\[
 \sum_{\substack{m\le X\\ \omega(m)\ge u}}\tau(m)^H
 \le z^{-u}X(1+\log X)^{\lceil zJ\rceil-1};
\]
indeed, $\omega(m)\ge u$ implies $z^{\omega(m)}\ge z^u$.  Choose
$z=u/B>1$.  Since $\lceil zJ\rceil-1\le zJ=u/\ell_X$, the logarithm of
the factor after $X$ is at most
\[
 -u\log\frac uB+\frac u{\ell_X}\ell_X
 =-u\log\frac uB+u.
\]
\end{proof}

Taking $H=0$ yields the following convenient large-$\omega$ estimate.

\begin{corollary}\label{cor:omega-tail}
Let $X\ge3$ and $u>\ell_X$.  Then
\[
 \#\{m\le X:\omega(m)\ge u\}
 \le X\exp\!\left(-u\log\frac u{\ell_X}+u\right).
\]
\end{corollary}

We shall also remove integers whose radical is much smaller than the integer itself.

\begin{lemma}[Radical defect]\label{lem:radical-defect}
There is an absolute constant $C$ such that
\[
 \sum_{n\le X}\left(\frac n{\rad(n)}\right)^{1/2}\le CX
 \qquad(X\ge1).
\]
Consequently, for every $Y\ge1$,
\[
 \#\left\{n\le X:\frac n{\rad(n)}>Y\right\}\le CXY^{-1/2}.
\]
\end{lemma}

\begin{proof}
Let
\[
 f(n)=\left(\frac n{\rad(n)}\right)^{1/2}.
\]
This is multiplicative, with $f(p^0)=f(p)=1$ and
$f(p^\nu)=p^{(\nu-1)/2}$ for $\nu\ge2$.  Put $g=f*\mu$, where $\mu$ is the M\"obius function.  Then $g$ is nonnegative and multiplicative, $g(p)=0$, and
\[
 g(p^\nu)=p^{(\nu-1)/2}-p^{(\nu-2)/2}\qquad(\nu\ge2).
\]
Moreover,
\[
 \sum_{\nu\ge2}\frac{g(p^\nu)}{p^\nu}\ll p^{-3/2},
\]
and therefore
\[
 \sum_p\sum_{\nu\ge2}\frac{g(p^\nu)}{p^\nu}<\infty.
\]
Since $f=\mathbfone*g$,
\[
 \sum_{n\le X}f(n)
 =\sum_{d\le X}g(d)\left\lfloor\frac Xd\right\rfloor
 \le X\sum_{d\ge1}\frac{g(d)}d\ll X.
\]
The final assertion follows from Markov's inequality.
\end{proof}

For the lower bound we use the standard multiplicative large sieve.  We state precisely the form needed here.

\begin{theorem}[Multiplicative large sieve]\label{thm:large-sieve}
For complex numbers $a_n$, and $N,Q\ge1$,
\[
 \sum_{q\le Q}\frac q{\varphi(q)}
 \sum_{\chi\, (\mathrm{mod}\,q)}^{*}
 \left|\sum_{n\le N}a_n\chi(n)\right|^2
 \ll (N+Q^2)\sum_{n\le N}|a_n|^2.
\]
Here the star denotes primitive Dirichlet characters, extended by zero on integers not coprime to the modulus.
\end{theorem}

This is the Bombieri--Davenport form of the multiplicative large sieve \cite{BombieriDavenport1969}; see also \cite[Section 7.4]{IwaniecKowalski2004}.  The implied constant in the displayed inequality is absolute.

We also record a short logarithmic-interval consequence of the prime number theorem.

\begin{lemma}[Primes in logarithmic intervals]\label{lem:pnt-layer}
Suppose $u\to\infty$ and $u^{-2}\le\delta\le1$.  Uniformly in this range,
\[
 \#\{p\text{ prime}:\e^{u-\delta}<p\le\e^u\}
 =(1+o(1))\frac{(1-\e^{-\delta})\e^u}{u}.
\]
\end{lemma}

\begin{proof}
The prime number theorem with its classical zero-free-region error gives
\[
 \pi(y)=\operatorname{li}(y)+O\bigl(y\e^{-c\sqrt{\log y}}\bigr)
\]
for some absolute $c>0$; see, for example, \cite[Chapter 6]{MontgomeryVaughan2007}.  Subtracting the formula at $y=\e^u$ and $y=\e^{u-\delta}$, the main term is
\[
 \int_{\e^{u-\delta}}^{\e^u}\frac{dt}{\log t}
 =\left(1+O\!\left(\frac1u\right)\right)
 \frac{(1-\e^{-\delta})\e^u}{u},
\]
uniformly for $u^{-2}\le\delta\le1$.
The error is $O(\e^u\e^{-c\sqrt u})$, which is $o(\delta\e^u/u)$ uniformly for $\delta\ge u^{-2}$.
\end{proof}

\section{A subset-product lemma}

The construction in the lower bound rests on the following Fourier lemma;
the random variables need not have identical distributions.  Let $G$ be
a finite abelian group, written multiplicatively, and let $\widehat G$ be
its character group.  For a
$G$-valued random variable $X$, write
\[
 \widehat\nu_X(\chi):=\mathbb E\chi(X).
\]

\begin{lemma}[Subset products hit the identity]\label{lem:subset-product}
Let $X_1,\ldots,X_m$ be independent $G$-valued random variables.  Suppose
\[
 \max_{1\le j\le m}\max_{\chi\ne1}
 |\widehat\nu_{X_j}(\chi)|\le\rho,
 \qquad m\rho\le1.
\]
Put
\[
 \Lambda:=\frac{2^m}{|G|}.
\]
There are absolute constants $\Lambda_0,C_1>0$ such that, whenever
$\Lambda\ge\Lambda_0$,
\[
 \mathbb P\left(
 \prod_{j\in J}X_j\ne1_G
 \text{ for every nonempty }J\subseteq\{1,\ldots,m\}
 \right)
 \le\frac{C_1}{\Lambda}.
\]
\end{lemma}

\begin{proof}
Let $Z_0$ count all subsets, including the empty subset, whose product is
the identity.  Thus $Z_0$ is a real nonnegative integer.  Character
orthogonality gives
\[
 Z_0=\frac1{|G|}\sum_{\chi\in\widehat G}
 \prod_{j=1}^m(1+\chi(X_j)).
\]
The principal character contributes $\Lambda$.  For every nonprincipal
character,
\[
 \left|\prod_{j=1}^m(1+\widehat\nu_{X_j}(\chi))\right|
 \le(1+\rho)^m\le\e,
\]
so
\begin{equation}\label{eq:EZ}
 \mathbb EZ_0=\Lambda+O(1).
\end{equation}

For the second moment,
\[
 \mathbb EZ_0^2
 =\frac1{|G|^2}\sum_{\chi,\psi\in\widehat G}
 \prod_{j=1}^m
 \left(1+\widehat\nu_{X_j}(\chi)
 +\widehat\nu_{X_j}(\psi)
 +\widehat\nu_{X_j}(\chi\psi)\right).
\]
We use four disjoint classes of pairs.
\begin{enumerate}[label=\textnormal{(\roman*)},leftmargin=2.5em]
\item The single pair $(1,1)$ contributes $\Lambda^2$.
\item There are $2(|G|-1)$ pairs for which exactly one of $\chi,\psi$
 is principal.  For each such pair the absolute value of the product is $O(2^m)$, so their
 normalized total contribution is $O(\Lambda)$.
\item There are $|G|-1$ pairs with $\chi\psi=1$ and both characters
 nonprincipal.  For each such pair,
 \[
  \left|\prod_{j=1}^m\left(2+\widehat\nu_{X_j}(\chi)
  +\widehat\nu_{X_j}(\chi^{-1})\right)\right|=O(2^m),
 \]
 and their normalized total contribution is again $O(\Lambda)$.
\item For every remaining pair, $\chi,\psi$, and $\chi\psi$ are all
 nonprincipal.  The corresponding product is at most
 $(1+3\rho)^m\le\e^3$ in absolute value, and all such pairs together
 contribute $O(1)$ after division by $|G|^2$.
\end{enumerate}
Hence
\[
 \mathbb EZ_0^2\le\Lambda^2+O(\Lambda+1).
\]
By \eqref{eq:EZ},
$(\mathbb EZ_0)^2\ge\Lambda^2-O(\Lambda+1)$, and therefore
\[
 \operatorname{Var}(Z_0)\ll\Lambda+1.
\]
If no nonempty subset product equals the identity, then $Z_0=1$.
For $\Lambda\ge\Lambda_0$ with $\Lambda_0$ sufficiently large,
$|1-\mathbb EZ_0|\ge\Lambda/2$.  Chebyshev's inequality now gives
\[
 \mathbb P(Z_0=1)\ll\frac1\Lambda,
\]
which proves the assertion after enlarging $C_1$.
\end{proof}

\section{The constructive lower bound}

Put
\[
 \alpha:=\log 2.
\]
For a large positive integer $r$, define
\begin{equation}\label{eq:lower-parameters}
 v:=\alpha(r-1)-10\log r,
 \qquad
 \delta:=\frac{8\alpha}{\log r},
 \qquad
 u_j:=v-(r-j)\delta\quad(1\le j\le r).
\end{equation}
\begin{remark}\label{rem:lower-constants}
The numerical constants $10$ and $8$ are convenient safety margins, not
optimized choices.  The term $10\log r$ ensures
$2^{r-1}/p_i\ge r^{10}$, so the union of the subset-product failure
probabilities is summable.  The choice $\delta\asymp1/\log r$ keeps the
layers wide enough for prime counting while making the accumulated
cleaning and interval-width losses lower order.
\end{remark}
Let
\begin{equation}\label{eq:prime-layers}
 \Pcal_j:=\{p\text{ prime}:\e^{u_j-\delta}<p\le\e^{u_j}\},
 \qquad M_j:=|\Pcal_j|.
\end{equation}
The intervals are pairwise disjoint.  Uniformly for $1\le j\le r$,
\[
 u_j=\alpha r+O\!\left(\frac r{\log r}+\log r\right),
\]
so $u_1\to\infty$ and, for all sufficiently large $r$,
$u_j^{-2}\le\delta\le1$.  Hence \cref{lem:pnt-layer} applies to every
layer with one uniform error term, and
\begin{equation}\label{eq:Mj-asymp}
 M_j=(1+o(1))\frac{(1-\e^{-\delta})\e^{u_j}}{u_j}.
\end{equation}
In particular,
\begin{equation}\label{eq:Mj-exponential}
 \min_j\log M_j=\alpha r-o(r).
\end{equation}

We next remove a negligible collection of target primes for which one of the layers has a large Fourier coefficient.

\begin{lemma}[Fourth-moment cleaning]\label{lem:fourth-moment}
Fix $j$.  For a prime $p\le\e^v$ and a nonprincipal character $\chi\pmod p$, put
\[
 S_{j,p}(\chi):=\sum_{q\in\Pcal_j}\chi(q).
\]
Let $\Bcal_j$ be the set of primes $p\le\e^v$ for which
\[
 |S_{j,p}(\chi)|>\frac{M_j}{20r}
\]
for at least one nonprincipal $\chi\pmod p$.  Then
\[
 |\Bcal_j|=\exp(o(r))
\]
uniformly in $j$.
\end{lemma}

\begin{proof}
Let $b_j(n)$ count ordered representations $n=q_1q_2$ with
$q_1,q_2\in\Pcal_j$, and put $b_j(n)=0$ outside this support.  Then
\[
 S_{j,p}(\chi)^2=\sum_n b_j(n)\chi(n),
\]
and unique factorization gives the exact identity
\[
 \sum_n b_j(n)^2=M_j+4\binom{M_j}{2}=2M_j^2-M_j\le2M_j^2.
\]
Every nonprincipal character modulo a prime is primitive.  Applying \cref{thm:large-sieve} with $Q=\lfloor\e^v\rfloor$ and $N=\lfloor\e^{2u_j}\rfloor$ to the sequence $b_j(n)$, and then restricting the resulting sum to prime moduli, gives
\[
 \sum_{p\le\e^v}\frac p{p-1}
 \sum_{\substack{\chi\, (\mathrm{mod}\,p)\\ \chi\ne\chi_0}}
 |S_{j,p}(\chi)|^4
 \ll (\e^{2u_j}+\e^{2v})M_j^2.
\]
Since $u_j\le v$, the large-sieve factor satisfies
\[
 N+Q^2\le \e^{2u_j}+\e^{2v}\le2\e^{2v},
\]
and since $p/(p-1)\ge1$, restriction from all moduli to prime moduli gives
\[
 \sum_{p\le\e^v}\sum_{\substack{\chi\, (\mathrm{mod}\,p)\\ \chi\ne\chi_0}}
 |S_{j,p}(\chi)|^4
 \ll \e^{2v}M_j^2.
\]
Consequently,
\[
 |\Bcal_j|
 \ll (20r)^4\frac{\e^{2v}}{M_j^2}.
\]
By \eqref{eq:Mj-asymp},
\[
 \log\left(\frac{\e^{2v}}{M_j^2}\right)
 \le2(v-u_j)+O(\log r+|\log\delta|)
 \le\frac{16\alpha r}{\log r}+O(\log r)=o(r).
\]
This proves the assertion.
\end{proof}

Let
\[
 \Bcal:=\bigcup_{j=1}^r\Bcal_j,
 \qquad
 \Pcal_j^*:=\Pcal_j\setminus\Bcal.
\]
By \cref{lem:fourth-moment}, $|\Bcal|\le r\exp(o(r))=\exp(o(r))$.
Together with \eqref{eq:Mj-exponential}, this gives, uniformly in $j$,
\begin{equation}\label{eq:bad-set-ratio}
 \frac{|\Bcal|}{M_j}
 \le\exp\bigl(-(\log 2+o(1))r\bigr).
\end{equation}
In particular $|\Bcal|=o(M_j/r)$, and hence
\begin{equation}\label{eq:clean-size}
 |\Pcal_j^*|=(1-o(1))M_j
\end{equation}
uniformly in $j$.  Moreover,
\[
 0\le-\sum_{j=1}^r\log\frac{|\Pcal_j^*|}{M_j}
 \le \frac{2r|\Bcal|}{\min_j M_j}=o(r).
\]
If $p\in\Pcal_i^*$, $j\ne i$, and $\chi\ne\chi_0\pmod p$, then, for
all sufficiently large $r$,
\begin{equation}\label{eq:clean-fourier}
 \left|
 \frac1{|\Pcal_j^*|}\sum_{q\in\Pcal_j^*}\chi(q)
 \right|
 \le
 \frac{M_j/(20r)+|\Bcal|}{M_j-|\Bcal|}
 \le\frac1{10r}.
\end{equation}

Choose independently and uniformly
\[
 p_j\in\Pcal_j^*\qquad(1\le j\le r),
\]
and put
\[
 n=p_1\cdots p_r.
\]
Fix $i$ and condition on the value of $p_i$.  The remaining variables
$p_j$ are still independent, because the original choice was a product
measure.  The intervals are disjoint, so $p_j\ne p_i$ for $j\ne i$;
therefore each $p_j$ is a nonzero element of
\[
 G_i:=(\mathbb Z/p_i\mathbb Z)^\times.
\]
The characters of $G_i$ are exactly the Dirichlet characters modulo
$p_i$ restricted to the units, and \eqref{eq:clean-fourier} verifies the
Fourier hypothesis of \cref{lem:subset-product} with $m=r-1$ and
$\rho=1/(10r)$.  Also, by \eqref{eq:lower-parameters},
\[
 p_i\le\e^v=\frac{2^{r-1}}{r^{10}},
\]
so
\[
 \Lambda_i:=\frac{2^{r-1}}{p_i-1}\ge r^{10}.
\]
Therefore
\[
 \mathbb P\left(
 \text{no nonempty product of the }p_j\ (j\ne i)
 \text{ is }1\pmod{p_i}
 \right)\ll r^{-10}.
\]
A union bound over $i$ gives total failure probability $O(r^{-9})$.  Hence
a proportion $1-o(1)$ of the tuples $(p_1,\ldots,p_r)$ have the property
that, for every $i$, a nonempty subset of the other selected primes has
product congruent to $1$ modulo $p_i$.  That product is a divisor $d>1$
of $n$, so it is a witness divisor for $p_i$.  Therefore every such
selected product $n$ belongs to $\Acal$.

We now count these integers.  Put
\begin{equation}\label{eq:Lr}
 L_r:=\sum_{j=1}^r u_j.
\end{equation}
Every selected product is at most $\e^{L_r}$, and the disjointness of the layers makes the representation unique.  Thus
\begin{equation}\label{eq:lower-count-product}
 A(\e^{L_r})\ge(1-o(1))\prod_{j=1}^r|\Pcal_j^*|.
\end{equation}
From \eqref{eq:Mj-asymp}, \eqref{eq:clean-size}, and the uniform product estimate after \eqref{eq:clean-size},
\[
 \log\prod_{j=1}^r|\Pcal_j^*|
 =L_r-\sum_{j=1}^r\log u_j+r\log(1-\e^{-\delta})+o(r).
\]
Uniformly in $j$, $u_j=\alpha r+O(r/\log r+\log r)$.  Hence
\[
 \sum_{j=1}^r\log u_j=r\log r+O(r),
\]
and, since $\delta=8\alpha/\log r$,
\[
 r\log(1-\e^{-\delta})=-r\log\log r+O(r).
\]
Also,
\begin{equation}\label{eq:Lr-asymp}
 L_r=\alpha r^2+O\left(\frac{r^2}{\log r}+r\log r\right).
\end{equation}
It follows that
\begin{equation}\label{eq:lower-subsequence}
 \log A(\e^{L_r})
 \ge L_r-r\log r-r\log\log r+O(r).
\end{equation}
By \eqref{eq:Lr-asymp},
\[
 r\log r
 =\left(\frac1{2\sqrt{\alpha}}+o(1)\right)\sqrt{L_r}\log L_r,
\]
while $r\log\log r=o(\sqrt{L_r}\log L_r)$.  Consequently,
\begin{equation}\label{eq:lower-at-Lr}
 A(\e^{L_r})
 \ge\e^{L_r}\exp\!\left(-\left(\frac1{2\sqrt{\log 2}}+o(1)\right)
 \sqrt{L_r}\log L_r\right).
\end{equation}
From the definitions one has the exact expression
\[
 L_r=\alpha r(r-1)-10r\log r-\frac{4\alpha r(r-1)}{\log r}.
\]
Consequently
\[
 L_{r+1}-L_r=2\alpha r+O\!\left(\frac r{\log r}+\log r\right),
\]
which is positive for all sufficiently large $r$ and is $O(r)=O(\sqrt{L_r})$.  Thus $(L_r)$ is eventually strictly increasing.  For arbitrary large $x$, choose the unique sufficiently large $r$ such that $L_r\le\log x<L_{r+1}$.  The gap $\log x-L_r$ is $O(\sqrt{\log x})=o(\sqrt{\log x}\log\log x)$, and in this range
\[
 \sqrt{L_r}\log L_r=(1+o(1))\sqrt{\log x}\log\log x.
\]
Moreover,
\[
 \e^{L_r}=x\exp\bigl(- (\log x-L_r)\bigr)
 =x\exp\bigl(-O(\sqrt{\log x})\bigr).
\]
Since $O(\sqrt{\log x})=o(\sqrt{\log x}\log\log x)$, monotonicity of $A$ and \eqref{eq:lower-at-Lr} yield the following.

\begin{theorem}[Lower bound]\label{thm:lower}
As $x\to\infty$,
\[
 A(x)\ge x\exp\!\left(-\left(\frac1{2\sqrt{\log 2}}+o(1)\right)
 \sqrt{\log x}\log\log x\right).
\]
\end{theorem}

\section{Canonical witnesses and deterministic homogeneous sequences}

For each $n\in\Acal$ and each prime $p\mid n$, fix the canonical witness
\begin{equation}\label{eq:canonical-witness}
 D_p=D_p(n):=\min\{d:d\mid n,\ d>1,\ d\equiv1\pmod p\}.
\end{equation}
The defining set is finite and nonempty because $n\in\Acal$, so the
minimum exists.  Necessarily $p\nmid D_p$.  Minimality is used only to make
the witness choice deterministic; no quantitative property of the minimum
is needed.

Write the distinct prime factors of $n$ in decreasing order as
\[
 p_1>p_2>\cdots>p_t.
\]
For each $1\le r\le t$ we now define a deterministic selected sequence.  Let $S_0=\{p_1,\ldots,p_r\}$.  Having defined a nonempty set $S_{i-1}$, let $q_i$ be its largest element and partition $S_{i-1}\setminus\{q_i\}$ into
\[
 S_{i-1}^{+}:=\{q:q\mid D_{q_i}\},
 \qquad
 S_{i-1}^{0}:=\{q:q\nmid D_{q_i}\}.
\]
If $|S_{i-1}^{+}|>|S_{i-1}^{0}|$, put $S_i=S_{i-1}^{+}$ and label row $i$ positive; otherwise put $S_i=S_{i-1}^{0}$ and label row $i$ zero.  Thus ties are always resolved in favor of the zero class.

\begin{lemma}[Canonical row-homogeneous sequence]\label{lem:homogeneous-sequence}
Let $r\ge1$ and put
\[
 s_r:=1+\lfloor\log_2r\rfloor.
\]
The deterministic procedure above produces distinct primes
\[
 q_1>q_2>\cdots>q_{s_r}
\]
with the property that, for every $i<s_r$, either every later $q_j$ divides $D_{q_i}$ or no later $q_j$ divides $D_{q_i}$.
\end{lemma}

\begin{proof}
Let $N_i:=|S_i|$.  At every step,
\[
 N_i\ge\left\lceil\frac{N_{i-1}-1}{2}\right\rceil.
\]
If $N_0\ge2^k$, induction on $i$ gives
$N_i\ge2^{k-i}$ for $0\le i\le k$: indeed, when
$N_{i-1}\ge2^{k-i+1}$, the displayed recurrence gives
$N_i\ge\lceil(2^{k-i+1}-1)/2\rceil=2^{k-i}$.  Thus the construction
makes $k+1$ selections.  With $k=\lfloor\log_2r\rfloor$, it reaches
$q_{s_r}$.  For $i<s_r$, the row label records which homogeneous class was retained;
no label is attached to the final selected row, since it has no later
selected prime.
\end{proof}

For $r\ge2$, put
\begin{equation}\label{eq:sr-hr}
 s=s_r,
 \qquad
 h=h_r:=\left\lfloor\frac{s_r}{2}\right\rfloor.
\end{equation}
Then
\begin{equation}\label{eq:hr-asymp}
 h_r=\frac{\log r}{2\log 2}+O(1).
\end{equation}

\section{Canonical compression and reconstruction}

We first isolate the valuation fact used in the fiber argument.

\begin{lemma}[One-bit valuation completion]\label{lem:valuation-bit}
Let $Q$ be squarefree, let $n=mQ$, and let $D\mid n$.  Put
$a=\gcd(D,m)$.  For every prime $q\mid Q$ define
\[
 \epsilon_q(D;m,Q):=v_q(D)-v_q(a).
\]
Then $\epsilon_q(D;m,Q)\in\{0,1\}$ and
\begin{equation}\label{eq:valuation-completion}
 D=a\prod_{q\mid Q}q^{\epsilon_q(D;m,Q)}.
\end{equation}
\end{lemma}

\begin{proof}
For $q\mid Q$ one has $v_q(n)=v_q(m)+1$.  Since $D\mid n$,
\[
 v_q(a)=\min\{v_q(D),v_q(m)\},
\]
so $v_q(D)-v_q(a)$ is $0$ or $1$.  At primes outside $Q$, the
exponents of $n$ and $m$ agree, and the exponent of $D$ is already
recorded in $a$.  This proves \eqref{eq:valuation-completion} prime by
prime.
\end{proof}

For reference, the $q$-adic bookkeeping is
\begin{center}
\begin{tabular}{@{}cccc@{}}
\toprule
$v_q(n)$ & $v_q(m)$ & $v_q(\gcd(D,m))$ & $v_q(D)$ \\
\midrule
$1$ & $0$ & $0$ & $\epsilon_q$ \\
$e+1$ & $e$ & $\min\{v_q(D),e\}$
    & $v_q(\gcd(D,m))+\epsilon_q$ \\
\bottomrule
\end{tabular}
\qquad $\epsilon_q\in\{0,1\}$.
\end{center}

Fix $n\in\Acal$ and $2\le r\le t=\omega(n)$.  Apply
\cref{lem:homogeneous-sequence}, and abbreviate $s=s_r$ and $h=h_r$.
Let
\begin{align*}
 I_+&:=\{1\le i\le h:\text{row }i\text{ is positive}\},\\
 I_0&:=\{1,\ldots,h\}\setminus I_+,
 & h_+&:=|I_+|.
\end{align*}
Since $s-h\ge h\ge h_+$, define
\begin{equation}\label{eq:deleted-slots}
 E:=I_0\cup\{h+1,\ldots,h+h_+\}
\end{equation}
and
\begin{equation}\label{eq:Qr-def}
 Q_r(n):=\prod_{j\in E}q_j,
 \qquad m_r(n):=\frac{n}{Q_r(n)}.
\end{equation}
Then $Q_r(n)$ is squarefree, $\omega(Q_r(n))=h$, and
\begin{equation}\label{eq:Qr-lower}
 Q_r(n)\ge p_r^h.
\end{equation}
A deleted slot $j\in E$ is \emph{exact} if $v_{q_j}(n)=1$, equivalently
$q_j\nmid m_r(n)$, and is \emph{visible} if $q_j\mid m_r(n)$, in which
case $v_{q_j}(n)\ge2$.

\begin{center}
\small
\begin{tabular}{@{}p{0.30\textwidth}ccp{0.32\textwidth}@{}}
\toprule
Slot type & Deleted? & Encoded in $\mathbf b$? & Recovery \\
\midrule
Positive first-half row & no & yes & already visible in $m$ \\
Visible zero-row slot & yes & yes & already visible in $m$ \\
Exact zero-row slot & yes & no & induction from $D_{q_i}-1$ \\
Visible deleted suffix slot & yes & yes & already visible in $m$ \\
Exact deleted suffix slot & yes & no & common product recovered by CRT \\
Undeleted suffix slot & no & yes & already visible in $m$ \\
\bottomrule
\end{tabular}
\normalsize
\end{center}

\begin{definition}[Half-row fiber record]\label{def:fiber-record}
Fix $x,t,r$ and $m$.  Let $n\le x$ lie in $\Acal$, satisfy
$\omega(n)=t$, and have $m_r(n)=m$.  Its $r$-record is
\[
 \mathcal R_r(n)=
 (I_+,\mathbf b,\mathbf a,\boldsymbol\epsilon,\boldsymbol\kappa),
\]
where:
\begin{enumerate}[label=\textnormal{(R\arabic*)},leftmargin=3.1em]
\item $I_+\subseteq\{1,\ldots,h\}$ is the positive-row set;
\item for $1\le j\le s$, let $b_j=q_j$ if $q_j\mid m$, and let
$b_j=0$ otherwise;
\item $a_i:=\gcd(D_{q_i},m)$ for $1\le i\le h$;
\item the rows of $\boldsymbol\epsilon$ are indexed by $1\le i\le h$,
and its columns by the elements of $E$ in increasing order; for $j\in E$,
\[
 \epsilon_{i,j}:=v_{q_j}(D_{q_i})-v_{q_j}(a_i)\in\{0,1\};
\]
\item if $i\in I_0$ and $b_i=0$, then $\kappa_i$ is the position of
$q_i$ in the increasing list of distinct prime divisors of
$D_{q_i}-1$; in every other row $\kappa_i=0$.
\end{enumerate}
\end{definition}

By a \emph{formal $r$-record} we mean any tuple with the following
coordinate ranges, whether or not it is realized by an integer $n$:
$I_+\subseteq\{1,\ldots,h\}$; each $b_j$ is either zero or a prime divisor
of $m$; each $a_i$ is a positive divisor of $m$; every
$\epsilon_{i,j}$ belongs to $\{0,1\}$; and every $\kappa_i$ is an integer
between zero and $\lfloor\log_2x\rfloor$.

The last entry is well-defined: since $q_i\mid D_{q_i}-1$ and
$D_{q_i}>1$, the positive integer $D_{q_i}-1$ has $q_i$ as a prime
divisor.  Moreover,
\begin{equation}\label{eq:kappa-bound}
 1\le\kappa_i\le\omega(D_{q_i}-1)\le\lfloor\log_2x\rfloor.
\end{equation}

\begin{lemma}[Number of records]\label{lem:record-count}
For fixed $x,t,r$ and positive integer $m$, the number of formal records is at most
\begin{equation}\label{eq:record-count-exact}
 2^h(1+\omega(m))^s\tau(m)^h2^{h^2}
 (1+\lfloor\log_2x\rfloor)^h.
\end{equation}
Consequently it is at most
\begin{equation}\label{eq:record-count-soft}
 \tau(m)^h
 \exp\!\left(C_{\mathrm{fib}}\bigl((\log(t+2))^2+
 \log(t+2)\log\log(3x)\bigr)\right)
\end{equation}
for an absolute constant $C_{\mathrm{fib}}$.
\end{lemma}

\begin{proof}
There are at most $2^h$ choices for $I_+$.  Each $b_j$ is zero or a
prime divisor of $m$, giving at most $(1+\omega(m))^s$ ordered vectors.
There are $\tau(m)^h$ choices for the $a_i$.  Since $|E|=h$, the
valuation matrix has at most $h^2$ bits.  Finally, each branch index is
bounded by \eqref{eq:kappa-bound}, and allowing zero in every row gives
the last factor.  Since $s,h=O(\log(t+2))$ and
$\omega(m)\le\log_2x$, the logarithm of all factors except
$\tau(m)^h$ has the asserted bound.
\end{proof}

The proof of injectivity is divided into three reconstruction lemmas.
At the beginning, every nondeleted selected prime and every visible
deleted prime is known from $\mathbf b$.  During zero-row recovery we
maintain the invariant that all exact deleted slots with smaller index
have already been recovered.

\begin{lemma}[Exact zero-row reconstruction]\label{lem:zero-row-recovery}
Suppose $i\in I_0$ and $b_i=0$.  Once every exact deleted slot $j<i$
has been recovered, the record determines $q_i$ uniquely.
\end{lemma}

\begin{proof}
The canonical witness $D_{q_i}$ is not divisible by $q_i$.  Because row
$i$ is zero, no selected prime $q_j$ with $j>i$ divides $D_{q_i}$.
Every prime factor outside the deleted product $Q_r(n)$ is already
captured, with its full valuation, by $a_i=\gcd(D_{q_i},m)$.  Hence the
only factors missing from $a_i$ are deleted primes in slots $j<i$, all
of which are either visible or have already been recovered.  By
\cref{lem:valuation-bit},
\begin{equation}\label{eq:zero-row-reconstruction-formula}
 D_{q_i}=a_i\prod_{\substack{j\in E\\j<i}}
 q_j^{\epsilon_{i,j}}.
\end{equation}
The right side is known.  The recorded index $\kappa_i$ then selects
$q_i$ uniquely among the distinct prime divisors of $D_{q_i}-1$.
\end{proof}

After applying \cref{lem:zero-row-recovery} successively in increasing
order, all exact zero-row slots are known.  Define
\[
 F:=\{j:h<j\le h+h_+,\ b_j=0\},
 \qquad P:=\prod_{j\in F}q_j.
\]

\begin{lemma}[Positive-row common product]\label{lem:positive-common-product}
For every $i\in I_+$, define
\begin{equation}\label{eq:positive-row-Ci}
 C_i:=a_i\prod_{j\in E\setminus F}q_j^{\epsilon_{i,j}}.
\end{equation}
Then $C_i$ is known from the record and the recovered zero rows, and
\begin{equation}\label{eq:positive-row-DPCi}
 D_{q_i}=P C_i.
\end{equation}
Moreover, $\gcd(C_i,q_i)=1$.
\end{lemma}

\begin{proof}
A positive first-half row contains every later selected prime, in
particular every suffix prime.  Each exact deleted suffix prime has
exponent one in $n$, and therefore occurs exactly once in every positive
row witness; their common contribution is $P$.  Every other deleted
prime is already known, and its possible one-step valuation excess over
$a_i$ is the recorded bit $\epsilon_{i,j}$.  At primes outside the
deleted product, $a_i$ contains the full valuation.  This proves
\eqref{eq:positive-row-DPCi} prime by prime.  Since a witness for $q_i$
is never divisible by $q_i$, while $q_i\nmid P$, we have
$q_i\nmid C_i$.
\end{proof}

\begin{lemma}[CRT recovery and size]\label{lem:crt-recovery}
The record determines $P$ uniquely.  More precisely,
\begin{equation}\label{eq:crt-uniqueness-cong}
 P\equiv C_i^{-1}\pmod{q_i}\qquad(i\in I_+),
\end{equation}
and, with $R:=\prod_{i\in I_+}q_i$,
\begin{equation}\label{eq:P-size}
 1\le P<R
\end{equation}
whenever $I_+\ne\varnothing$.  If $I_+=\varnothing$, then $P=1$.
\end{lemma}

\begin{proof}
The congruences follow from
\[
 D_{q_i}=PC_i\equiv1\pmod{q_i}
 \qquad\text{and}\qquad
 \gcd(C_i,q_i)=1.
\]  If $I_+=\varnothing$, then $h_+=0$, so there are no
deleted suffix slots and $P=1$.

Assume $I_+\ne\varnothing$.  Write
$I_+=\{i_1<\cdots<i_{h_+}\}$ and
$F=\{f_1<\cdots<f_k\}$, where $k\le h_+$.  If $k=0$, then $P=1<R$,
because $R$ is a nonempty product of primes.  Assume henceforth that
$k\ge1$.  Since every suffix index is larger than $h$ and every
$i_\ell\le h$, we have $f_\ell>i_\ell$ for $1\le\ell\le k$.  The
selected primes decrease with their indices, so
$q_{f_\ell}<q_{i_\ell}$.  Therefore
\[
 P=\prod_{\ell=1}^k q_{f_{\ell}}
 <\prod_{\ell=1}^k q_{i_{\ell}}
 \le\prod_{i\in I_+}q_i=R.
\]
The Chinese remainder theorem produces a unique integer $P_0$ satisfying
$0\le P_0<R$ and \eqref{eq:crt-uniqueness-cong}.  Every residue
$C_i^{-1}\pmod{q_i}$ is nonzero, hence $P_0\ne0$.  Since the actual $P$
satisfies the same congruences and $1\le P<R$, we have $P=P_0$.
\end{proof}

\begin{proposition}[Reconstruction of a fixed-prefix fiber]
\label{prop:record-reconstruction}
Fix $x,t,r,m$.  A record arising from an integer $n\le x$ in $\Acal$
with $\omega(n)=t$ and $m_r(n)=m$ determines $n$ uniquely.
\end{proposition}

\begin{proof}
First read every visible selected prime from $\mathbf b$.  Then process
$i=1,\ldots,h$ and apply \cref{lem:zero-row-recovery} whenever
$i\in I_0$ and $b_i=0$.  This recovers all exact zero-row primes.
Next \cref{lem:positive-common-product,lem:crt-recovery} recover the
integer $P$.

Take the prime factorization of the already determined integer $P$.
Write $F=\{f_1<\cdots<f_k\}$.  A record arising from an actual integer has
$P$ squarefree and, if $k>0$,
\[
 P=\ell_1\cdots\ell_k,
 \qquad \ell_1>\cdots>\ell_k
\]
for distinct primes $\ell_j$.  In every valid preimage the selected primes
satisfy $q_1>\cdots>q_s$, so the only possible assignment is
\[
 q_{f_j}=\ell_j\qquad(1\le j\le k).
\]
The already known visible suffix primes give consistency checks for the
strict inequalities between adjacent slots.  If squarefreeness, the factor
count, or any ordering check fails, the formal record has no preimage.
Thus every deleted slot is known, and
\[
 Q_r(n)=\prod_{j\in E}q_j,
 \qquad n=mQ_r(n).
\]
Hence two preimages with the same record coincide.  Factorization is used
only to define an injective counting decoder; no computational-complexity
claim is made.
\end{proof}

\begin{remark}[Schematic slot pattern]\label{rem:schematic-record}
Let $h=3$, $I_+=\{1,3\}$, and $I_0=\{2\}$.  Then $h_+=2$ and
$E=\{2,4,5\}$.  In the schematic pattern
\[
 \underbrace{q_1}_{\text{positive, visible}}
 >\underbrace{q_2}_{\text{zero, exact}}
 >\underbrace{q_3}_{\text{positive, visible}}
 >\underbrace{q_4}_{\text{suffix, exact}}
 >\underbrace{q_5}_{\text{suffix, visible}},
\]
$q_2$ is recovered first from its zero-row witness.  The only remaining
unknown suffix factor is then $P=q_4$, which is recovered from the two
congruences modulo $q_1$ and $q_3$.  The known neighbouring prime $q_5$
fixes the unique suffix placement.  The general proof is the same, with
several exact factors recovered simultaneously through their product.
\end{remark}

\begin{proposition}[Fixed-prefix fiber bound]\label{prop:fixed-prefix-fiber}
Let $n\le x$ range over integers in $\Acal$ with $\omega(n)=t$, and fix
$2\le r\le t$.  For every positive integer $m$, the number of $n$ satisfying
$m_r(n)=m$ is at most
\[
 \tau(m)^{h_r}
 \exp\!\left(C_{\mathrm{fib}}\bigl((\log(t+2))^2+
 \log(t+2)\log\log(3x)\bigr)\right).
\]
\end{proposition}

\begin{proof}
The record map is injective by \cref{prop:record-reconstruction}, and
its range is bounded by \cref{lem:record-count}.
\end{proof}

\begin{lemma}[One-prime compression]\label{lem:one-prime}
Let $n\in\Acal$, let $p=P^+(n)$, put $Q=p$ and $m=n/p$.  For a fixed positive integer $m$,
the number of possible $n\le x$ is at most
\[
 \tau(m)(1+\lfloor\log_2x\rfloor).
\]
\end{lemma}

\begin{proof}
The witness $D_p$ is not divisible by $p$, hence $D_p\mid m$.  Once
$D_p$ is fixed, $p$ is a prime divisor of the positive integer
$D_p-1<x$, which has at most $\lfloor\log_2x\rfloor$ distinct prime
divisors.
\end{proof}

For $n\in\Acal$ with $t=\omega(n)\ge2$, define
\[
 \sigma_1(n):=\log p_1,
 \qquad \sigma_r(n):=h_r\log p_r\quad(2\le r\le t).
\]
Let $\rho(n)$ be the least index at which the maximum score is attained.
If $\rho(n)=1$, put $Q(n)=p_1$; otherwise put
$Q(n)=Q_{\rho(n)}(n)$.  Finally set
\[
 \mathfrak m(n):=\frac{n}{Q(n)}.
\]
These deterministic conventions make $n\mapsto(Q(n),\mathfrak m(n))$
a single-valued map.

For $t\ge2$ define
\begin{equation}\label{eq:Ht-Jt}
 H_t:=\left\lceil\frac{\log(t+2)}{2\log 2}\right\rceil+3,
 \qquad J_t:=2^{H_t}.
\end{equation}
The additive constant $3$ is a harmless safety margin ensuring
$h_r\le H_t$ uniformly for every $r\le t$ and absorbing the one-prime
branch and small values of $t$.  In particular, $J_t\ll\sqrt{t+2}$.

\begin{proposition}[Canonical compression fiber]\label{prop:canonical-fiber}
There is an absolute constant $C_{\mathrm{fib}}$ such that, for every $x\ge3$,
$t\ge2$, and positive integer $m$,
\begin{align*}
 &\#\{n\le x:n\in\Acal,\ \omega(n)=t,\ \mathfrak m(n)=m\}\\
 &\quad\le\tau(m)^{H_t}
 \exp\!\left(C_{\mathrm{fib}}\bigl((\log(t+2))^2+
 \log(t+2)\log\log(3x)\bigr)\right).
\end{align*}
Moreover, $Q(n)$ is squarefree and $\omega(Q(n))\le H_t$.
\end{proposition}

\begin{proof}
Partition the fiber according to $r=\rho(n)$.  For $r=1$, use
\cref{lem:one-prime}.  Since $H_t\ge1$ and, after enlarging
$C_{\mathrm{fib}}$,
\[
 1+\lfloor\log_2x\rfloor
 \le\exp\!\bigl(C_{\mathrm{fib}}\log(t+2)\log\log(3x)\bigr),
\]
this branch has the displayed bound.  For $2\le r\le t$, use the
fixed-prefix fiber bound and $h_r\le H_t$.  Summing at most $t$ branches
only adds $\log t$ to the exponent.  Squarefreeness and the bound for
$\omega(Q(n))$ follow from the construction.
\end{proof}

\section{Weighted-prefix extraction}

The canonical score maximizer forces $Q(n)$ to be large.

\begin{lemma}[The reciprocal staircase sum]\label{lem:hr-sum}
As $t\to\infty$,
\[
 \sum_{r=2}^t\frac1{h_r}
 =\left(2\log 2+o(1)\right)\frac t{\log t}.
\]
\end{lemma}

\begin{proof}
By \eqref{eq:hr-asymp}, uniformly for $r\ge3$,
\[
 \frac1{h_r}=\frac{2\log 2}{\log r}+O\!\left(\frac1{(\log r)^2}\right).
\]
The $r=2$ term contributes $O(1)$.  The standard integral comparison
gives
\[
 \sum_{3\le r\le t}\frac1{\log r}
 =\frac t{\log t}+O\!\left(\frac t{(\log t)^2}\right),
 \qquad
 \sum_{3\le r\le t}\frac1{(\log r)^2}
 =O\!\left(\frac t{(\log t)^2}\right),
\]
which proves the assertion.
\end{proof}

We use $o_t(1)$ for a quantity tending to zero as $t\to\infty$, uniformly
over all integers $n$ with $\omega(n)=t$.

\begin{proposition}[Large canonical compression divisor]\label{prop:weighted-prefix}
Let $n\in\Acal$, let $p_1>\cdots>p_t$ be its distinct prime factors, and put
\[
 W:=\log\rad(n)=\sum_{j=1}^t\log p_j.
\]
For the canonical divisor $Q(n)$ defined above,
\begin{equation}\label{eq:weighted-Q}
 \log Q(n)\ge
 \left(\frac1{2\log 2}-o_t(1)\right)
 W\frac{\log t}{t}.
\end{equation}
Equivalently, for every $\eta>0$ there exists $T_\eta$ such that, whenever $t\ge T_\eta$,
\begin{equation}\label{eq:weighted-Q-explicit}
 \log Q(n)\ge
 \frac{1-\eta}{2\log 2}
 W\frac{\log t}{t}.
\end{equation}
\end{proposition}

\begin{proof}
Let
\[
 M:=\max\left\{\log p_1,\max_{2\le r\le t}h_r\log p_r\right\}.
\]
Then $\log p_1\le M$ and $\log p_r\le M/h_r$ for $r\ge2$.  By \cref{lem:hr-sum},
\[
 W\le M\left(1+\sum_{r=2}^t\frac1{h_r}\right)
 =\left(2\log 2+o(1)\right)M\frac t{\log t}.
\]
Hence
\[
 M\ge\left(\frac1{2\log 2}-o(1)\right)W\frac{\log t}{t}.
\]
The canonical index $\rho(n)$ attains the maximum.  If $\rho(n)=1$, then $\log Q(n)=M$.  If $\rho(n)=r\ge2$, then \eqref{eq:Qr-lower} gives $\log Q(n)\ge h_r\log p_r=M$.  This proves both formulations.  In particular, the compression removes
only $O(\log t)$ distinct prime factors, while its logarithm is of order
$W\log t/t$.
\end{proof}

\section{The upper bound}

Put
\begin{equation}\label{eq:L-notation}
 \begin{gathered}
 L:=\log x,\qquad L_2:=\log L,\qquad L_3:=\log L_2,\\
 S:=\sqrt L\,L_2,\qquad c_0:=\frac1{2\sqrt{\log 2}}.
 \end{gathered}
\end{equation}

The limiting parameters in this section are chosen in the following order.
In the proof of the upper bound we first fix the final error tolerance
$\eps>0$, then choose $\eta$, then the threshold $T_\eta$ supplied by
\cref{prop:weighted-prefix}, then the auxiliary cutoffs $\delta,C,T$, and
finally let $x$ tend to infinity.  All $o(1)$-terms below are uniform after
the preceding parameters have been fixed.

\subsection{Regular integers and the compression count}

We first remove integers far below $x$ or with a large radical defect.  Fix an absolute constant $K_0$ with $K_0/2>c_0+1$.  By \cref{lem:radical-defect}, the number of integers $n\le x$ satisfying either
\[
 n\le x\e^{-K_0S}
 \quad\text{or}\quad
 \frac n{\rad(n)}>\e^{K_0S}
\]
is
\begin{equation}\label{eq:irregular}
 O(x\e^{-K_0S/2}).
\end{equation}
Call the remaining integers \emph{regular}.  For a regular integer,
\begin{equation}\label{eq:regular-radical}
 W:=\log\rad(n)\ge L-2K_0S=(1-o(1))L.
\end{equation}

Fix $\eta\in(0,1/4)$.  Choose $T_\eta$ as in \eqref{eq:weighted-Q-explicit}.  For $x$ sufficiently large in terms of $\eta$, regularity gives $W\ge(1-\eta)L$.  Hence every regular $n\in\Acal$ with $t=\omega(n)\ge T_\eta$ satisfies
\begin{equation}\label{eq:Mt-def}
 \log Q(n)\ge M_{t,\eta}:=
 \frac{(1-\eta)^2L\log t}{2t\log 2}.
\end{equation}
Since $Q(n)$ has at most $H_t$ distinct prime factors,
\[
 \omega(\mathfrak m(n))\ge t-H_t.
\]
The canonical fiber bound therefore gives the uniform counting inequality
\begin{equation}\label{eq:compression-count}
 \begin{split}
 N_t^{\mathrm{reg}}(x)
 &:=\#\{n\le x:n\in\Acal,\ n\text{ regular},\ \omega(n)=t\}\\
 &\le E(x,t)
 \sum_{\substack{m\le x\e^{-M_{t,\eta}}\\
                   \omega(m)\ge t-H_t}}
 \tau(m)^{H_t},
 \end{split}
\end{equation}
where we define
\begin{equation}\label{eq:E-small}
 E(x,t):=\exp\!\left(C_{\mathrm{fib}}\bigl((\log(t+2))^2+
 \log(t+2)\log\log(3x)\bigr)\right).
\end{equation}
\subsection{The critical range}

We record the two uniform critical estimates in a quantifier-explicit form.

\begin{lemma}[Uniform critical rates]\label{lem:critical-rates}
Fix $0<\delta<C<\infty$, $\eta\in(0,1/4)$, and $\zeta>0$.  There exists $x_0=x_0(\delta,C,\eta,\zeta)$ such that, for every $x\ge x_0$ and every integer $t$ satisfying
\[
 \delta\sqrt L\le t\le C\sqrt L,
 \qquad
 \lambda:=\frac t{\sqrt L},
\]
one has
\begin{align}
 \#\{n\le x:\omega(n)\ge t\}
 &\le x\exp\!\left(-\left(\frac\lambda2-\zeta\right)S\right),
 \label{eq:ordinary-rate}\\
 N_t^{\mathrm{reg}}(x)
 &\le x\exp\!\left(-\left(\frac\lambda4+
 \frac{(1-\eta)^2}{4\lambda\log 2}-\zeta\right)S\right).
 \label{eq:compression-rate}
\end{align}
\end{lemma}

\begin{proof}
Enlarge $x_0$, if necessary, so that $\delta\sqrt L\ge T_\eta$ for every $x\ge x_0$.  Then every $t$ in the stated interval is large enough for the compression lower bound \eqref{eq:Mt-def}.
Apply \cref{cor:omega-tail} with $X=x$ and $u=t$.  Uniformly for $\lambda\in[\delta,C]$,
\[
 \log\frac{t}{\log(1+L)}
 =\frac12L_2-L_3+O_{\delta,C}(1).
\]
The $t$, $\log(1+L)$, and $O(L_3)$ terms are $o(S)$ uniformly, which proves \eqref{eq:ordinary-rate} once $x$ is large enough.

For \eqref{eq:compression-rate}, put
\[
 X:=x\e^{-M_{t,\eta}},
 \qquad
 u:=t-H_t.
\]
In the present range $M_{t,\eta}=O(S)=o(L)$, so $X\ge3$ for large $x$.  Uniformly for $\lambda\in[\delta,C]$, the definitions give
\[
\begin{aligned}
 H_t&=\frac{\log t}{2\log 2}+O(1),\\
 J_t&=t^{1/2}\e^{O(1)}=L^{1/4}\e^{O_{\delta,C}(1)},\\
 u&=t+O(\log t),\qquad \log X=L+O(S),\\
 \log\frac{t}{\log(1+L)}
 &=\frac12L_2-L_3+O_{\delta,C}(1),\\
 \log\frac{u}{J_t\log(1+\log X)}
 &=\frac14L_2-L_3+O_{\delta,C}(1).
\end{aligned}
\]
Moreover
\[
 \sqrt L\,L_3=o(S),\qquad L_2^2=o(S),\qquad J_t L_2=o(t)
\]
uniformly in the same range.  Hence $J_t\log(1+\log X)=o(t)$ and $u=t-H_t>J_t\log(1+\log X)$ for all sufficiently large $x$.  Hence \cref{lem:restricted-moment} applies.  Write
$\ell_X:=\log(1+\log X)$.  Its exponent is
\begin{align*}
 -u\log\frac{u}{J_t\ell_X}+u
 &=-\bigl(\lambda\sqrt L+O(L_2)\bigr)
 \left(\frac14L_2-L_3+O_{\delta,C}(1)\right)\\
 &\qquad+\lambda\sqrt L+O(L_2)\\
 &=-\left(\frac\lambda4+o_{\delta,C}(1)\right)S.
\end{align*}
Furthermore,
\begin{align*}
 M_{t,\eta}
 &=\frac{(1-\eta)^2L\log t}{2t\log 2}\\
 &=\left(\frac{(1-\eta)^2}{4\lambda\log 2}
 +o_{\delta,C}(1)\right)S.
\end{align*}
Finally, $\log E(x,t)=O(L_2^2)=o(S)$ uniformly.  Combining these three
displays and enlarging $x_0$ makes the total normalized error smaller
than $\zeta$.
\end{proof}

\subsection{The noncritical ranges}

\begin{lemma}[Noncritical values of $\omega$]\label{lem:noncritical}
Fix $\eps>0$ and $\eta\in(0,1/4)$.  There exist $0<\delta<1<C$, an integer $T\ge T_\eta$, and $x_1=x_1(\eps,\eta,\delta,C,T)$ such that, for $x\ge x_1$, the number of regular $n\le x$ in $\Acal$ satisfying either
\[
 \omega(n)<\delta\sqrt L
 \quad\text{or}\quad
 \omega(n)>C\sqrt L
\]
is at most
\[
 x\exp(-(c_0+\eps)S).
\]
\end{lemma}

\begin{proof}
Choose $C$ so large that $C/2>c_0+3\eps$, and put
$u_C:=\lfloor C\sqrt L\rfloor+1$.  Every integer with
$\omega(n)>C\sqrt L$ has $\omega(n)\ge u_C$.  Applying
\cref{cor:omega-tail} with $u=u_C$ gives the required estimate for this
range.

Choose $\delta>0$ so small that
\begin{equation}\label{eq:delta-choice}
 \frac{(1-\eta)^2}{4\delta\log 2}>c_0+3\eps.
\end{equation}
Now choose $T\ge\max\{T_\eta,3\}$ large enough that \eqref{eq:weighted-Q-explicit} is available for every $t\ge T$ and
\[
 \frac{(1-\eta)^2\log t}{2t\log 2}<\frac12
 \qquad(t\ge T).
\]
Thus $M_{t,\eta}<L/2$ throughout the following range, so every divisor-sum parameter is at least one.  If $T\le t\le\delta\sqrt L$, then $\log t/t$ is decreasing and \eqref{eq:Mt-def} gives, uniformly,
\[
 M_{t,\eta}
 \ge\left(\frac{(1-\eta)^2}{4\delta\log 2}+o(1)\right)S.
\]
Dropping the condition on $\omega(m)$ in \eqref{eq:compression-count} and using \cref{lem:local-moment} and \cref{lem:dk-sum},
\[
 \sum_{m\le x\e^{-M_{t,\eta}}}\tau(m)^{H_t}
 \le x\e^{-M_{t,\eta}}(1+L)^{J_t-1}.
\]
Uniformly for $t\le\delta\sqrt L$,
\[
 J_t L_2\ll L^{1/4}L_2=o(S),
\]
and \eqref{eq:E-small} is also $o(S)$.  By \eqref{eq:delta-choice}, each of these terms has the required bound.
There are at most $O(\sqrt L)$ admissible values of $t$, and their total
cost is $\exp(O(\log L))=\exp(o(S))$.

It remains to treat $0\le t<T$.  The case $t=0$ consists only of $n=1$.  If $t=1$ and $n>1$, then $n$ is a prime power and every divisor $d>1$ is divisible by its unique prime factor, so $n\notin\Acal$.  Fix $2\le t<T$.  From regularity,
\[
 p_1\ge\rad(n)^{1/t}
 \ge\exp\!\left(\frac{L-2K_0S}{t}\right).
\]
Apply \cref{lem:one-prime}.  With
\[
 X_t:=x\exp\!\left(-\frac{L-2K_0S}{t}\right),
\]
we obtain, by \cref{lem:dk-sum} with $k=2$,
\[
 N_t^{\mathrm{reg}}(x)
 \le (1+\lfloor\log_2x\rfloor)
 \sum_{m\le X_t}\tau(m)
 \le (1+\lfloor\log_2x\rfloor)X_t(1+\log X_t).
\]
Thus the number of such $n$ is $x^{1-1/t+o(1)}$, uniformly for each of the finitely many $t<T$, which is much smaller than $x\exp(-(c_0+\eps)S)$.
\end{proof}

\subsection{Optimization}

\begin{theorem}[Upper bound]\label{thm:upper}
As $x\to\infty$,
\[
 A(x)\le x\exp\!\left(-\left(\frac1{2\sqrt{\log 2}}+o(1)\right)
 \sqrt{\log x}\log\log x\right).
\]
\end{theorem}

\begin{proof}
Fix $\eps>0$.  Choose $\eta\in(0,1/4)$ so small that
\[
 (1-\eta)c_0\ge c_0-\frac{\eps}{4}.
\]
Apply \cref{lem:noncritical} with error parameter $\eps/2$ and this fixed
$\eta$; let $\delta,C,T$ and $x_1$ be the resulting constants.  Next
apply \cref{lem:critical-rates} with these fixed $\delta,C,\eta$ and with
$\zeta=\eps/8$, obtaining a threshold $x_0$.  We henceforth assume that
$x\ge\max\{x_0,x_1\}$ and is large enough that
\[
 \log(C\sqrt L+1)\le\frac{\eps}{8}S.
\]

For $\lambda>0$, put
\[
 F_\eta(\lambda):=\max\left\{
 \frac\lambda2,
 \frac\lambda4+\frac{(1-\eta)^2}{4\lambda\log 2}
 \right\}.
\]
The two expressions are equal at
$\lambda_0=(1-\eta)/\sqrt{\log 2}$.  The second expression is decreasing
on $(0,\lambda_0]$, while the first is increasing on
$[\lambda_0,\infty)$.  Therefore
\begin{equation}\label{eq:Feta-min}
 \inf_{\lambda>0}F_\eta(\lambda)
 =\frac{1-\eta}{2\sqrt{\log 2}}
 =(1-\eta)c_0.
\end{equation}
For each integer $t$ in the critical interval, take the better of
\eqref{eq:ordinary-rate} and \eqref{eq:compression-rate}.  Its normalized
exponential rate is at least
\[
 F_\eta(t/\sqrt L)-\zeta
 \ge c_0-\frac{\eps}{4}-\frac{\eps}{8}.
\]
There are at most $C\sqrt L+1$ such values of $t$, so their union has size
at most
\[
 x\exp\!\left(-\left(c_0-\frac{\eps}{2}\right)S\right).
\]
The noncritical regular integers are bounded by
$x\exp(-(c_0+\eps/2)S)$ by \cref{lem:noncritical}; the irregular integers
are $O(x\exp(-K_0S/2))$, and $K_0/2>c_0+1$.  Thus, after increasing the
threshold for $x$ once more,
\[
 A(x)\le x\exp\bigl(-(c_0-\eps)S\bigr).
\]
Since $\eps>0$ was arbitrary, the upper bound follows.  Together with
\cref{thm:lower}, it proves \cref{thm:main}.
\end{proof}

\begin{remark}
The two directions use different manifestations of the same binary
entropy.  The construction has $2^{r-1}$ possible subset products for a
target prime; the upper bound repeatedly halves a witness-incidence set
and then deletes half of the resulting homogeneous sequence.
\end{remark}

\section*{Acknowledgements}
The author thanks the authors and contributors of the PrimeNumberTheoremAnd
project, whose \texttt{MediumPNT} theorem supplies the analytic input of the
formalisation, the Mathlib community, and Thomas Bloom for creating and
maintaining \href{https://www.erdosproblems.com}{erdosproblems.com}.

\end{document}